\documentclass[referee, sn-mathphys,Numbered]{sn-jnl}

\usepackage{graphicx}%
\usepackage{multirow}%
\usepackage{amsmath,amssymb,amsfonts}%
\usepackage{amsthm}%

\usepackage{mathrsfs}%
\usepackage[title]{appendix}%
\usepackage{xcolor}%
\usepackage{textcomp}%
\usepackage{manyfoot}%
\usepackage{booktabs}%
\usepackage{algorithm}%
\usepackage{algorithmicx}%
\usepackage{algpseudocode}%
\usepackage{listings}%
\usepackage{hyperref}
\usepackage{mathtools,ulem}
\usepackage{soul} 
\newtheorem{dfn}{Definition}%
\newtheorem{thm}{Theorem}
\newtheorem{lem}{Lemma}

\graphicspath{{NoisyFixedPointsFigures/}}

\begin{document}

\title[Random Logistic Map]{Noisy fixed points: stability of the invariant distribution of the random logistic map}

\author*[1]{\fnm{Kimberly} \sur{Ayers}}\email{kayers@csusm.edu}

\author[2]{\fnm{Ami} \sur{Radunskaya}}\email{aradunskaya@pomona.edu}
\equalcont{These authors contributed equally to this work.}

\affil*[1]{\orgdiv{Department of Mathematics}, \orgname{California State University San Marcos}, \orgaddress{\street{333 S. Twin Oaks Valley Rd.}, \city{San Marcos}, \postcode{92096}, \state{CA}, \country{USA}}}

\affil[2]{\orgdiv{Department of Mathematics and Statistics}, \orgname{Pomona College}, \orgaddress{\street{610 College Ave.}, \city{Claremont}, \postcode{91711}, \state{CA}, \country{USA}}}

\abstract{
The full family of discrete logistic maps has been widely studied both as a canonical example of the period-doubling route to chaos, and as a model of natural processes. In this paper we present a study of the stochastic process described by iterations of the logistic map with a random parameter.  In addition to bringing together previously known results, we present a proof that the unique invariant measure of the process in the chaotic regime is asymptotically stable.  
}

\keywords{stochastic quadratic map; discrete chaos; Markov process; asymptotic stability} 


\maketitle

\section{Introduction}

The discrete logistic map: 
$$f_\lambda(x) = \lambda x (1 - x),$$
is well known as a model of self-limiting population growth, whose dynamic properties are described in the early review paper by Robert May \cite{May}.  In this highly influential paper (over 6,600 citations as of November, 2023), May states that, while the fine structure of the bifurcation diagrams in unimodal maps are fascinating, what is more important in practice is the statistical properties of these maps.  In his opinion, ``What is called for is an effectively stochastic description of the dynamics", and that these descriptions of evolving  continuous densities will still capture some salient features of the maps themselves.  
{
The bifurcation diagram of the logistic family shows how the long term dynamics of the map change as the parameter, $\lambda$, is varied.  What May is referring to in this challenge is a description of the distribution of $x$-values, the ``continuously evolving densities", as the parameter $\lambda$ varies stochastically. 
}
{In practice, we might want to model a scenario in which  the intrinsic growth rate of a population fluctuates over a specific range by letting the parameter $\lambda$ be a random variable with support in a small interval.  
}

{
In the literature on random dynamical systems, a stochastically-varying parameter is known as ``parametric noise" to distinguish it from ``additive noise".  This distinction is important from a modeling perspective: parametric noise models stochastic features of the environment, be it reproduction rates that vary due to temperature or resources, or economic growth rates, which might vary in response to world events, politics, or other economic pressures.  Additive noise, on the other hand, models uncertainty in the model predictions, so that the model becomes: $X_{n+1} = f(X_n) + \epsilon_n$, where $\epsilon_n$ is the additive noise which is typically independent of the previous observation, $X_n$, and of the law of evolution, $f(\cdot)$.    In this paper we consider parametric noise, modeling uncertainty in the model parameters.  Mathematically, parametric noise gives rise to some interesting questions, such as the uniqueness and stability of invariant measures.  We note that, in some cases, additive noise which is state-dependent (i.e. $\epsilon_n$ depends on $X_n$) can be reinterpreted as parametric noise, so the distinction between the two is not always clear. See, for example,
}
\cite{Klebaner1998} 
{ 
which studies quasi-stationary distributions of the logistic map with state-dependent additive noise. For a theoretical discussion of additive noise in the context of random dynamical systems, we refer the reader to 
}
\cite{Lasota}.  
{
The logistic map with additive noise has also been studied in physical and computational contexts.  For example, in 
}
\cite{Baptista1996} 
{
this type of random logistic map 
 been studied as a model of physical processes with random ``kicks".  Other examples (a far from complete list!)  include 
 }
 \cite{Sato2018}, 
 {
 which discusses how bifurcations of the logistic map are affected by additive noise,
 }
 \cite{Sano2020}, 
 {
 which describes how synchronization time can be shortened with additive noise, and
 }
 \cite{Hamzi2023}
 {
 which studies the random logistic map with additive noise in the context of algorithmic probability.
 }

Since May's remarks appeared, many models 
{
involving parametric noise
}have appeared in the ecology, biology and economics literature.  In \cite{Haskell2005}, the authors address the Neubert conjecture by looking at a stochastically-perturbed Beverton-Holt model of population growth.  A continuous version of the stochastic logistic model is studied in \cite{Descheemaeker2020}, where they show that these systems can be used to describe experimental results on microbial growth. Three different economic models described by randomly perturbed deterministic maps are studied in \cite{Satoh2001},  \cite{nishimura2004} and \cite{Bhattacharya2004}.  These examples are by no means exhaustive, but they do give a sense of the ubiquity and impact of these models.  In most cases, the goal is to first show that the map has a unique invariant measure, since the uniqueness of this measure is essential if one is to say something about the long-term behavior of the system.  

Along these same lines, Kolmogorov suggested starting with a deterministic dynamical system, and adding a stochastic perturbation in such a way that the resulting process has a unique invariant measure.  This stochastic version of the original deterministic system is sometimes called the {\it Kolmogorov measure} of the dynamical system \cite{Kifer1986}.

In this paper we present a study of the stochastic process described by iterations of the logistic map with a randomly-varying parameter,  
{
which we will refer to as ``the stochastic logistic map".
Our contribution to the literature on the stochastic logistic map  is that, in
}
addition to bringing together previously known results, we present a proof that the unique invariant measure of the process is asymptotically stable.  

\section{Background and Definitions} \label{sec:background}

{
In this section we present some technical background, notation, and definitions for the reader who is new to the concepts of stochastic dynamical systems, transition probabilities and invariant measures.
}

Given a state space $S$, one can consider either a deterministic or a stochastic (or random) discrete dynamical system on $S$.  A deterministic dynamical system is given by a function $f:S\rightarrow S$, where, given an initial state $x_0\in S$, one can compute future states recursively:
$$x_{n+1}=f(x_n).$$
Because, given $x_n$, the value of $x_{n+1}$ can be computed exactly, this is a {\it deterministic} process, and the set of iterates, $x_0, x_1, x_2, \dots = \{x_n\}_{n=0}^\infty$, is called the {\it trajectory} of the initial value $x_0$ under the dynamical system given by the function $f$.  A natural question is: given an initial value, $x_0$, what is the long-term behavior of the sequence $x_0, x_1, x_2, \dots$?  If the function $f$ is continuous, and the sequence converges to a finite value, $x^*$, then it must be a fixed point, i.e. $x^*=f(x^*)$.  So one of the first questions a dynamicist asks when encountering a dynamical system is: where are the fixed points?   
Other questions of interest include determining the stability of fixed points - whether nearby trajectories converge toward,  or away from,  fixed points - and looking at the limiting behavior of non-fixed points. 

In a stochastic dynamical system, we consider a sequence of random variables: $\{X_n \}_{n = 0}^\infty$.  Given a value for the random variable $X_n$ in the state space, the value of $X_{n+1}$ cannot be determined exactly because the law of evolution is subject to some noise.  We are then interested in the sequence of distributions of these random variables.  In particular, we are interested in the existence of an \emph{invariant distribution}: a distribution of $X_n$ such that $X_{n+1}$ has the same distribution.  If such a distribution exists, we are also interested in its stability; that is, whether other distributions converge to this invariant distribution in some manner. We define the stochastic logistic map and the notions of invariant measures and their stability more precisely below.

{
In this paper, we consider the state space $S=[0,1]$ and the particular choice of parameter space 
$\Lambda=(3.87,4) \subset [0,4]$.  With these choices, we are ensured that any point in $S$ is mapped to a point in $S$, and that all of the logistic maps are in the ``chaotic regime".  This will be important in the proofs in Section
}
\ref{sec:proof}.

We consider the logistic Markov process 
\begin{equation} \label{eq:RandomLogistic}
  X_{n+1} = \lambda_{n+1}X_n(1-X_n),  
\end{equation}
where   $\{\lambda_i\}_{i\geq 1}$ are independent, $\Lambda$-valued random variables, independent of $X_n$ with identical distribution $Q$. The initial state $X_0$ is an $S$-valued random variable, independent of the $\lambda$'s. Let $A \subset S$. Then $\{X_n\}$ is an $S$-valued Markov process with a continuous state space, $S$, and  transition probabilities   
\begin{equation*}
 \hbox{Prob}\{X_{1} \in A | X_0 = x\} =  p(x,A)= Q\left(\frac{A}{x(1-x)}\right).
\end{equation*}
In words: the probability that $X_1$ is in a given set $A$, given $X_0 = x$, is the probability of the set of $\lambda$-values that map $x$ to a point in $A$.  The measure $p$ is called a (one-step) {\it transition probability} on the state space, $S$, and we can write $p(x,A) = \int_S {\bf 1}_A \, p(x,dy)$.

Similarly, we write the $n$th order transition probability as
$$ \hbox{Prob}\{X_n \in A | X_0 = x\} = p^n(x, A).$$

Suppose we have a distribution, $\mu$, on the state space, $S$, which gives the distribution of $X_n$.  When we apply the process \eqref{eq:RandomLogistic}
to subsets of $S$, we induce a new distribution, $P \mu$ on the state space as follows.  If $A \subset S$, then we define 
\begin{equation} \label{eq:Pdef} P\mu (A) = \int_S p(x,A) \mu(dx) .
\end{equation}
So the operator $P$ maps probability measures on $S$ to probability measures on $S$.  We can now define an {\it invariant measure} under the operator, $P$, which corresponds to a fixed point in the deterministic case.

\begin{dfn}
If $P \mu (A) = \mu(A)$ for all $A \in {\cal B}(S)$ then $\mu$ is an invariant measure under the operator $P$.
\end{dfn}

What 
{
we would
}
like is the {\it asymptotic stability} of the (unique) invariant measure, $\pi$.  In other words, for any probability measure $\mu$ on $S$, when is it the case that 
$$ \sup_{A \in {\cal B}(S)}
|P^n \mu (A) - \pi(A)| \rightarrow 0 \quad \hbox{ as } n \rightarrow \infty ?$$
To answer this question, we need the notions of {\it irreducibility} and {\it aperiodicity}.

The following definitions of irreducibility
and recurrence are 
 from \cite{Bhattacharya2007}, Def. 9.2
\begin{dfn}
Let $\mathcal{B}(S)$ denote the Borel subsets of $S$. 
\begin{enumerate}
\item[(i)]
A Markov process $\{X_n\}_{n=0}^\infty$ is {\bf irreducible} with respect to a non-zero $\sigma$-finite reference measure $\nu$ (or $\nu$-irreducible for short) if, for every $x\in S$ and every $B\in\mathcal{B}(S)$ with $\nu(B)>0$, there is $n\geq 1$ such that 
{
$p^n(x,B)>0$. 
}
\item[(ii)] The measure $\psi$ is a {\bf maximal irreducibility measure} for the Markov process $\{ X_n \}$ if the process is $\psi$-irreducible and it satisfies the following conditions:
\begin{enumerate}
\item For any other measure $\phi$, the process is $\phi$-irreducible iff $\psi \succ \phi$;
\item if $\psi(A) = 0$ then $\psi\{x : Prob(\eta_A < \infty | X_0 = x) > 0 \} = 0 $.
\end{enumerate}
\end{enumerate}
\end{dfn}
In other words: a Markov process is $\nu$-irreducible if there is a positive probability that every element in the state space eventually gets mapped to any set of positive $\nu$ measure.

If a process is $\nu$-irreducible, then it has a maximal irreducibility measure, $\psi$ (\cite{Meyn},Proposition 4.2.2). In the literature, it is common to use $\psi$ to denote a maximal irreducibility measure for the Markov process.  We will follow this convention.
\begin{dfn}
\begin{enumerate}
\item[(i)] A set $A$ is \textbf{$\nu$-recurrent} if $\nu(A) > 0$ and 
$$ \sum_{n=0}^\infty p^n(x,A) = \infty $$
for every $x \in S$.   A process is $\nu$-recurrent if every set of positive $\nu$-measure is $\nu$-recurrent. 

\textup{In other words, a process is \(\nu\)-recurrent if, starting from any point in the state space, the expected total number of visits to every set of positive \(\nu\)-measure is infinite.
Note: some authors omit the reference measure, and just say ``recurrent",  when discussing $\psi$-irreducible Markov chains.  The assumption is that the reference measure is $\psi$, a maximal irreducibility measure.
}


\item [(ii)]
A set $A \in {\cal B}(S)$ is {\bf Harris recurrent} if 
$$ \mbox{Prob}(X_n \in A \hbox{ i.o. } | X_0 \in A ) = 1$$
A $\psi$-irreducible process is {\bf Harris recurrent} if  every set of positive $\psi$-measure is Harris recurrent.
Note that ``i.o." stands for ``infinitely often". Since the process is $\psi$-irreducible, this is equivalent to:
$$\hbox{Prob}(\eta_A < \infty | X_0 = x) = 1$$
for every $x \in S$ and every measurable $A$ such that $\psi(A) > 0$.
\end{enumerate}
\end{dfn}
\noindent  We remark here that $\nu$-recurrent implies $\nu$-irreducible, but the converse is not true.  Typically irreducibility is proved first, and then recurrence.




The following theorem relates an invariant distribution $\pi$ with the maximal irreducibility measure $\psi$ for recurrent chains.
\begin{thm}[\label{thm:piequivpsi}\cite{Meyn},Theorem 10.4.9]Suppose a chain $\{X_n\}$ is recurrent.  Then the unique (up to constant multiples) invariant measure $\pi$ for $\{X_n\}$ is equivalent to $\psi$.
\end{thm}
Note that in the event that the invariant measure is finite, this then implies the existence of a unique probability measure, which we will from this point use $\pi$ to denote. 


To get asymptotic stability of the invariant measure, we need two more conditions: Harris irreducibility and aperiodicity.

\begin{dfn}
Suppose that the Markov process $\{ X_n \}$ is $\psi$-irreducible and admits an invariant probability measure.  Then $\{ X_n \}$ is  {\bf positive}.  If the process is Harris recurrent and positive, then it is called a {\bf positive Harris} chain, or {\bf positive Harris recurrent}.
\end{dfn}

\noindent Positive processes are automatically recurrent:
\begin{thm}[\label{positive}\cite{Meyn}, 10.1.1] If a chain $\{X_n \}$ is positive then it is ($\psi$)-recurrent.
\end{thm}
To this end, a positive chain is sometimes called {\bf positive recurrent}.\\
The following theorems give conditions for when a $\psi$-irreducible process is Harris recurrent. They require a few definitions.

\begin{dfn}[\cite{Athreya_Dai},pg 598]
A Markov chain $\{X_n\}$ has the {\bf Feller property} if
$$E[h(X_1) | X_0 = x_n) \mapsto E(h(X_1) | X_0 = x)$$
for every sequence $x_n \rightarrow x$ and $h$ a bounded continuous function on $S$.
\end{dfn}
In other words, $p(x_n, \cdot)$ converges to the distribution $p(x, \cdot)$ whenever $x_n \rightarrow x$.

\begin{dfn}[\cite{MeynTweedie1992}]
A kernel $T$ is a {\bf continuous component} of a function $K:(S, {\cal B}) \mapsto \mathbb{R}^+$ if (i) for $A \in {\cal B}$ the function $T(\cdot, A)$ is lower semi-continuous; (ii) for all $x \in S$ and $A \in {\cal B}$ the measure $T(x, \cdot)$ satisfies $K(x,A) \ge T(x,A)$.
\end{dfn}

\begin{dfn}[\cite{Meyn},pg 131]
For a probability $a$ on $\mathbb{Z}_+$ define the Markov transition function $K_a = \sum a(i) p^i$.  
A Markov chain is a {\bf $T$-chain} if, for some $a$, the Markov transition function $K_a$ admits a continuous component $T$ which is non-trivial for all $x \in S$.
\end{dfn}
\noindent
Note that ``non-trivial for $x$" means that $T(x,S) > 0$.

\begin{thm}[\label{tchain}\cite{Meyn}, Theorem 6.0.1 (iii)] If a Markov process is a $\psi$-irreducible Feller chain such that supp $\psi$ has non-empty interior, then it is a $\psi$-irreducible $T$-chain. 
\end{thm}

\begin{thm}[\label{Harris}\cite{Meyn}, Theorem 9.0.2] 
A $\psi$-irreducible $T$-chain $\{X_n\}$ is Harris recurrent if and only if $Prob\{ X_n \rightarrow\infty | X_0 = x\}=0$ for each $x\in S$, where $X_n \rightarrow \infty$ refers to the event where the trajectory visits each compact set only finitely often.
\end{thm}

%

\begin{dfn} \label{def:aperiodic}  A Markov process is said to be strongly aperiodic with respect to a $\sigma$-finite probability measure $\nu$ on $S$ if there is a set $A_0\subset S$ with $\nu(A_0)=1$ and an $c >0$ such that  one has
\begin{enumerate}
\item for all $x \in S$, $Prob(\eta_{A_0} < \infty | X_0 = x) = 1$, i.e.  $A_0$ is recurrent; 
\item for all $x\in A_0$, $A$ a measurable subset of $S$,
$p(x,A)\geq c \nu(A)$.
\end{enumerate}
\end{dfn}


\noindent
Let $\{X_n\}$ be a Markov process on a state space $S$, and ${\mathcal B}$ the Borel subsets of $S$.
\begin{thm} [\cite{Nummelin}, Corollary 6.7]
\label{thm:thrm3BlueBook}
Suppose $\{X_n\}$ is a positive Harris recurrent and strongly aperiodic Markov process with respect to some nonzero $\sigma$-finite measure $\nu$, with stationary distribution $\pi$.  Then for any initial distribution, $\mu$, of $X_0$:
$$\lim_{n\rightarrow\infty} \delta(P^n\mu,\pi) =0$$
where $\delta$ is the total variation distance 
$$\delta(\mu_1,\mu_2)=\sup_{B\in\mathcal{B}}|\mu_1(B)-\mu_2(B)|.$$  
\end{thm}
\noindent


As an example of the ideas presented in this section, let us consider a version of the stochastic logistic map as examined in \cite{Rao1993}.  In this paper, Bhattacharya and Rao consider the stochastic logistic map
$$X_{n+1}=\lambda_{n+1} X_n(1-X_n)$$ 
where $\lambda$ is distributed according to a Bernoulli distribution, and takes two possible values, $\alpha$ and $\beta$ (where $\alpha<\beta$).  Thus, the process switches randomly between two quadratic maps.  Bhattacharya and Rao demonstrate that if $0\leq \alpha < \beta \leq 1$, then $\delta_0$, the Dirac measure supported on 0, is the unique invariant probability on [0,1].  If $1<\alpha<\beta \leq 2$ then there is a unique invariant probability on (0,1), and this probability is nonatomic.  If $2<\alpha<\beta<1+\sqrt{5}$ and $\alpha\in[8/\beta(4-\beta),\beta)$, then there exists a unique invariant probability on (0,1) which is nonatomic and has its support contained in $[1/2,(1+\sqrt{5})/4]$. \\




\section{Proof of the asymptotic stability of the invariant measure in the chaotic regime.} \label{sec:proof}


In Section 2, we introduced the Markov chain results that will be used to establish the uniqueness and asymptotic stability of an invariant distribution. In this section, we apply these results to the stochastic logistic map with parameters supported in the chaotic regime. We first review the behavior of the logistic family as the parameter, $\lambda$, is varied.  We next recall a result of Athreya and Dai establishing the existence of an invariant probability measure for this process, and then prove that this invariant measure is unique and asymptotically stable. 

\subsection{Orbits of the logistic family.}

Recall that the deterministic logistic family: $f(x,\lambda) =f_{\lambda}(x) = \lambda x (1 - x)$ undergoes a period-doubling cascade as $\lambda$ increases from 1 to $\lambda_{2^\infty} \approx$ 3.56995.   Note that here iteration occurs only with respect to $x$; that is $f^{(m)}(x,\lambda)=f(f^{m-1}(x,\lambda),\lambda)$, where $\lambda$ remains fixed.  For some higher values of $\lambda$, the family exhibits stable periodic orbits of odd period, following the Sarkowski sequence, culminating with a stable period-3 orbit when $\lambda = 1 + \sqrt{8}$.  Interspersed between the values of $\lambda$ with stable periodic orbits are a set of large measure of parameter values for which there are no stable periodic orbits, and almost all initial conditions give rise to orbits with no finite period.  Thus, $\lambda_{2^\infty}$ marks the ``onset of chaos".  As $\lambda$ increases beyond $1 + \sqrt{8}$, the family of maps again experiences a period-doubling cascade, with stable orbits of period 6, 12, 24, 48, \dots, culminating at $\lambda_{3^\infty} \approx 3.8495$.  

When $\lambda$ is greater than this value, the maps exhibit chaotic behavior on a Cantor set of parameter values, with three ``chaotic bands" of orbits merging into one chaotic band at $\lambda \approx 3.857$.  
In what follows, we are interested in randomly perturbing the values of $\lambda$ in this chaotic regime.  To make sure we are well within the parameter range for which there is generically a single band chaotic attractor, we restrict  $\lambda \in (3.87, 4)$.  
The logistic family is well-studied in the literature, and more details on its dynamics can be found in, for example, \cite{Holmgren1996}.

\subsection{The stochastic logistic map and the existence of an invariant probability measure.}
We consider the stochastic logistic map, i.e. the Markov process given by 

\begin{equation}\label{SLMap}
X_{n+1}=\lambda_{n+1}X_n(1-X_n)
\end{equation}
where $\{\lambda_i\}_{i\geq 1}$ are i.i.d. random variables distributed uniformly on (3.87,4). For the remainder of this section, we restrict the state space to \(S=(0,1)\). The initial state $X_0$ is a random variable taking values in (0,1), with distribution independent of the $\lambda$'s.

Athreya and Dai demonstrate that the Markov process $\{X_n\}$ given by Equation \eqref{SLMap} has an invariant probability distribution.
\begin{thm}[\cite{Athreya_Dai}, Theorem 2] \quad Let $E[\ln \lambda_1] > 0$ and
$E[ |\ln(4 - \lambda)|] < \infty$.  Then there exists an invariant probability $\pi$ such that $\pi\{ 0 \} = 0$.
\end{thm}
\noindent
In our case, $\lambda \sim {\cal U}(3.87, 4)$ and it is straightforward to check the conditions:
$$ E[\ln \lambda_1] = \frac{1}{.13}\int_{3.87}^4 \ln \lambda d \lambda \approx 1.3699 > 0 $$ 
and
$$ E[ |\ln(4 - \lambda)|] = \frac{1}{.13} \int_{3.87}^4 |\ln(4 - \lambda)| \, d \lambda \approx 0.7376 < \infty $$
Thus, the process $\{ X_n \}$ has a non-trivial invariant probability measure.

Knowing that there is an invariant probability measure allows us to apply results from Section \ref{sec:background} to show that this measure is asymptotically stable.

\subsection{Asymptotic stability of the invariant probability distribution.}
We structure this section as follows: 
we first show the irreducibility of the Markov process in the chaotic regime using Lemmas \ref{lemma1} and \ref{lem:lemma2}, which together with the existence of the invariant measure demonstrates that the process is positive, and thus recurrent, by Theorem \ref{positive}. In Theorem \ref{harrisrec}, we demonstrate that the process is positive Harris recurrent.  This, along with strong aperiodicity, gives the asymptotic stability of the invariant measure by Theorem \ref{thm:thrm3BlueBook}.

The following lemma uses the existence of a dense orbit and continuity of the iterated functions to show that every $x\in(0,1)$ has a positive probability of entering any open set in finite time.  This lemma will be used to prove the irreducibility of $\{X_n\}$. 
\begin{lem}\label{lemma1}
Let $A$ be a nonempty open subset of $(0,1).$ For all $x\in (0,1)$ there is an $M$ such that $p^M(x,A) = Prob(X_M \in A | X_0 = x)  > 0 $.
\end{lem}

\begin{proof}
It is known \cite{shi2007} that there is a value $\alpha\in (3.87,4)$ such that $f(x,\alpha)=\alpha x(1-x)$ exhibits topological transitivity; that is $f(x,\alpha)$ contains an orbit that is dense in $0,\alpha/4)$. Furthermore, since $A\subset(0,1)$, we can choose $\alpha$ large enough so that $A\cap (0,\alpha/4)$ is nonempty. Call this dense orbit $\Gamma.$ Let $x\in (0,1)$ Note then that, for all $x\in (0,1)$, the set $\{\lambda x(1-x):3.87< \lambda <4\}$ is an open interval intersecting $(0,\alpha/4)$, and therefore must contain a point on $\Gamma.$  Therefore, there exists $\beta\in(3.87,4)$ such that $y=f(x,\beta)$ is on $\Gamma$. 

Note then, that since the orbit of $y$ is dense for parameter $\alpha$, there exists $M$ such that $f^(M)(y,\alpha)\in A$.  Then, $f^M(f(x,\beta),\alpha)\in A$.  Note further that $f(\ldots ,f(x,\lambda_1),\cdots,\lambda_{M+1})$ is a polynomial in $M+2$ variables and is thus continuous everywhere.  Therefore, since $f^M(f(x,\beta),\alpha)\in A$, and $A$ is an open set, there is a set of positive normalized Lebesgue measure $U$ in $(3.87,4)^{M+1}$ containing $(\beta,\underbrace{\alpha,\; \cdots \;,\alpha}_{\hbox{$M$ times}})$ such that for all $(\lambda_1,\cdots,\lambda_{M+1}) \in U$, $f(\ldots ,f(x,\lambda_1),\cdots,\lambda_{M+1})\in A$. This means that
$$Prob(X_{M+1}\in A|X_0=x)>0.$$

\end{proof}

The following lemma is used later to show strong aperiodicity as well as irreducibility.
\begin{lem} \label{lem:lemma2}
Let $S=(0,1)$.
\begin{equation*}
  X_{n+1}=\lambda_{n+1}X_n(1-X_n),\,\,\,n\geq0
\end{equation*}
with $\{\lambda_n\}_{n\geq 1}$ being i.i.d., $\lambda_i\sim \mathcal{U}(3.87,4)$, and $X_0$ an independent $S$-valued r.v. Let $\lambda_{min}=3.87$, and let $A_0= (1-\frac{1}{\lambda_{min}},1-\frac{1}{4})$ (that is, $A_0$ is the set of fixed points for the deterministic logistic maps corresponding to the possible parameter $\lambda$ values). Let $\phi$ be normalized Lebesgue measure on $A_0$.  Then there exists $c >0$ such that for all $x\in A_0$, and $A$ a measurable subset of $(0,1)$, we have $p(x,A)\geq c \phi(A).$

\end{lem}

\begin{proof}
Let $\lambda_{min}=3.87$, and let $A_0= (1-\frac{1}{\lambda_{min}},1-\frac{1}{4})$ (that is, $A_0$ is the set of fixed points for the deterministic logistic maps corresponding to the possible parameter $\lambda$ values).  Let $L_1$ be the length of $A_0$ (that is, $L_1=\frac{1}{\lambda_{min}} - \frac{1}{4}$ ) and let $L_2$ be the length of interval $(\lambda_{min},4)$ (that is, $L_2 = 4 - \lambda_{\min}$).  Let $\phi$ be normalized Lebesgue measure on $A_0$. Then $A_0$, by definition, has measure 1, and for any measurable set $A$
:$$
\phi(A) = \frac{1}{L_1}\int_{y\in (A \cap A_0)}dy$$
Fix $x\in A_0$ and $A \subset S$.  We want to show that $p(x,A) \ge c \, \phi(A)$ for some positive constant, $c$ (Definition \ref{def:aperiodic}).   
Since $\phi(A)=\phi(A\cap A_0)$ it suffices to demonstrate the above result for sets $A\subset A_0$.

Notice for any $A\subset A_0$ and  with $\phi(A)>0$, and $x\in A_0$,
$$\underset{x}{\sup}\left\{\frac{A}{x(1-x)}\right\}<\frac{1-\frac{1}{4}}{\frac{1}{4}\left(1-\frac{1}{4}\right)}=4$$ 
and 

$$
\underset{x}{\inf}\left\{
\frac{A}{x(1-x)}
\right\}
>
\frac{1-\frac{1}{\lambda_{\min}}}
{\frac{1}{\lambda_{\min}}\left(1-\frac{1}{\lambda_{\min}}\right)}
=\lambda_{\min}=3.87
$$

\noindent Therefore, for any $A\subset A_0$ and $x\in A_0$, $\frac{A}{x(1-x)}\subset (3.87,4).$

\noindent Given $x\in A_0$ and $A\subset A_0$, we therefore have:
\begin{eqnarray*}
p(x,A) &=& Prob\left(\lambda\in \frac{A}{x(1-x)}\right)\\
&=& \displaystyle \frac{1}{L_2}\int_{\frac{A}{x(1-x)}}d\lambda\\
&=& \frac{1}{L_2}\int_A\frac{1}{x(1-x)}dy\\
&=&\frac{1}{L_2}\frac{1}{x(1-x)}\int_A dy\\
&=&\frac{1}{L_2}\cdot\frac{1}{x(1-x)}\cdot L_1\phi(A)\\
&\geq& 4\cdot \frac{L_1}{L_2}\phi(A) \\
& = & \frac{1}{\lambda_{\min}} \phi(A) \, ,
\end{eqnarray*}
since $x(1-x)\leq \frac{1}{4}$ for all $x\in[0,1]$.  Thus, letting $c = \displaystyle{\frac{1}{\lambda_{\min}}}$ gives us the desired result.  
\end{proof}

\begin{thm}\label{irred}
Let $S=(0,1)$, and set $A_0$ be defined as above. As before, define the Markov process on $S$:
\begin{equation*}
  X_{n+1}=\lambda_{n+1}X_n(1-X_n),\,\,\,n\geq0
\end{equation*}
with $\{\lambda_n\}_{n\geq 1}$ being i.i.d., $\lambda_i\sim \mathcal{U}(3.87,4)$, and $X_0$ an independent $S$-valued r.v. Then $\{X_n\}_{n=0}^\infty$ is irreducible with respect to normalized Lebesgue measure on $A_0$.
    
\end{thm}

\begin{proof}
Let \(x\in(0,1)\), and let \(B\) be a measurable set with \(\phi(B)>0\). Since \(A_0\) is a nonempty open subset of \((0,1)\), Lemma \ref{lemma1} implies that there exists \(M\ge1\) such that
\[
p^M(x,A_0)>0.
\]Therefore,
\[
\begin{aligned}
p^{M+1}(x,B)
&=\int_Sp(y,B)\,p^M(x,dy)\\
&\ge\int_{A_0}p(y,B)\,p^M(x,dy)\\
&\ge c\phi(B)p^M(x,A_0)>0,
\end{aligned}
\]where the final inequality follows from Lemma \ref{lem:lemma2}. Thus, $p^{M+1}(x,B)=Prob(X_{M+1}\in B|X_0=x)>0$ and $\{X_n\}_{n=0}^\infty$ is irreducible with respect to normalized Lebesgue measure on $A_0$,
\end{proof}
Now that we have shown the process is irreducible with respect to some nontrivial measure, by Proposition 4.2.2 in \cite{Meyn}, there exists a maximal irreducibility measure $\psi$.  Since the process admits an invariant probability measure \(\pi\) with \(\pi((0,1))=1\) by [16], the process is positive. By Theorem \ref{positive}, it is therefore \(\psi\)-recurrent. Theorem \ref{thm:piequivpsi} then implies that its invariant probability measure \(\pi\) is equivalent to the maximal irreducibility measure \(\psi\). Thus, the process is recurrent with respect to both \(\psi\) and \(\pi\). We now show the stronger result that the process is positive Harris recurrent.
\begin{thm}\label{harrisrec}
As before, define the Markov process on $S = (0,1)$:
\begin{equation*}
  X_{n+1}=\lambda_{n+1}X_n(1-X_n),\,\,\,n\geq0
\end{equation*}
with $\{\lambda_n\}_{n\geq 1}$ being i.i.d., $\lambda_i\sim \mathcal{U}(3.87,4)$, and $X_0$ an independent $S$-valued r.v.  Then $\{X_n\}_{n=0}^\infty$ is positive Harris recurrent. 
\end{thm}
\begin{proof}
As discussed in \cite{Athreya_Dai}, the stochastic logistic map exhibits Feller continuity. By Theorem \ref{irred}, the process is $\phi$-irreducible. Let $A_0$ be the set of fixed points as defined in Lemma \ref{lem:lemma2}. Since the irreducibility measure $\phi$ is Lebesgue measure restricted to $A_0$ and $\nu\ll \psi$, the maximal irreducibility measure, every nonempty open subset of $A_0$ (and thus, $A_0$ itself) has positive $\psi$ measure. This demonstrates that $A_0\subset\text{supp } \psi$, and thus $\text{supp }\psi$ has nonempty interior. By Theorem \ref{tchain}, the process is a $\psi$-irreducible $T$-chain. 

We next apply Theorem \ref{harrisrec}.
For our random logistic process, we can show that, given $x\in S$, $Prob\{X_n \rightarrow\infty | X_0 = x \} =0$. Fix $x$ and consider the interval $[\varepsilon, 1-\varepsilon]$, where $\varepsilon>0$ is less than $\frac{1}{4}$.   Note that if $X_0>1-\varepsilon$, then $X_1<4\varepsilon(1-\varepsilon) < 1 - \varepsilon$ with probability 1.
Now consider the case $X_0<\varepsilon$. 
Note that,
for all $\lambda\in(3.87,4)$, $f(x,\lambda)=\lambda x(1-x)$ is necessarily greater than $x$.  In fact,
for $\varepsilon < \frac{1}{4}$, we have $3.87(1-\varepsilon)>1$ which  ensures that
$$X_n < \varepsilon \Rightarrow X_{n+1}=\lambda_{n+1} X_n(1-X_n)>3.87(1-\varepsilon)X_n=:cX_n$$with $c>1$.  Hence, while the trajectory remains below $\varepsilon$, it grows by at least a factor of $c$, giving us a uniform multiplicative lower bound, and so the trajectory must leave $(0,\varepsilon)$ after finitely many steps.  Let $X_n$ be such that $X_n<\varepsilon$ and $X_{n+1}\geq\varepsilon$. Since $\varepsilon < \frac{1}{4}$, 
$X_n < \frac{1}{4}$ and 
$$X_{n+1} < 4 \frac{1}{4} \left({1 - \frac{1}{4}}\right) = \frac{3}{4} < 1 - \varepsilon . $$
Thus the next iterate lands inside $[\varepsilon,1-\varepsilon]$ with probability 1.  Thus, regardless of the starting value or parameter sequence, every trajectory must return to the compact interval $[\varepsilon, 1-\varepsilon]$ infinitely often.  Consequently, no trajectory can visit every compact subset of $S$ only finitely often, and hence $Prob\{X_n \rightarrow\infty| X_0 = x\}=0$ for every $x\in(0,1).$ We then conclude by Theorem 4 that the process is (positive) Harris recurrent.

\end{proof}


We conclude this section by demonstrating that the process is strongly aperiodic, and thus, by Theorem \ref{thm:thrm3BlueBook} the invariant distribution is asymptotically stable with respect to the total variation norm.

\begin{thm}
Let $S=(0,1)$.
\begin{equation*}
  X_{n+1}=\lambda_{n+1}X_n(1-X_n),\,\,\,n\geq0
\end{equation*}
with $\{\lambda_n\}_{n\geq 1}$ being i.i.d., $\lambda_i\sim \mathcal{U}(3.87,4)$, and $X_0$ an independent $S$-valued r.v. Then $\{X_n\}_{n=0}^\infty$ has an invariant distribution that is asymptotically stable with respect to the total variation norm. 
\end{thm}

\begin{proof}
Because $\{X_n\}_{n=0}^\infty$ is positive Harris recurrent, it is $\psi$-recurrent. Since $\psi$ dominates $\phi$, it is also $\phi$-recurrent. This together with Lemma \ref{lem:lemma2} gives that the process is strongly aperiodic with respect to $\phi$.  By Theorem \ref{thm:thrm3BlueBook}, this implies that $\{X_n\}$ has an invariant distribution that is asymptotically stable with respect to the total variation norm.

\end{proof}

\section{Simulations and Discussion}




All code that was used to generate the figures in this section can be found at the GitHub repository here: \href{https://github.com/kdayers/NoisyFixedPoints}{https://github.com/kdayers/NoisyFixedPoints}.
Suppose $X$ is a random variable taking values in $[0,1]$.  If, for all $x$, 
$ Pr\{X \le x\} = Pr \{ f(X) \le x \} $ then we say that the distribution of $X$ is {\it invariant under $f$}.  In this way, we can discuss invariant distributions for the deterministic logistic map by looking at distributions invariant under $f_{\lambda}(x) = \lambda x ( 1 - x) $.   In \cite{ulam1947}, Ulam and von Neumann proved that the continuous invariant distribution for the deterministic logistic map with $\lambda = 4$ is a beta distribution with $\alpha = \beta = 0.5$.  The density for this distribution is symmetric on the interval (0,1), and goes to infinity at the endpoints of its support.  To be specific, the invariant distribution when $\lambda = 4$ is given by the density function 
\begin{equation} \label{eq:betadist}
f^*(x)= \begin{cases} \frac{1}{\pi\sqrt{x(1-x)}} & 0 < x < 1 \\ 0 & \hbox{ else} 
\end{cases}
\end{equation}
See Figure \ref{betadist} for a comparison of this beta density and an empirical distribution.  

\begin{figure}
\begin{center}
    \includegraphics[height =2in]{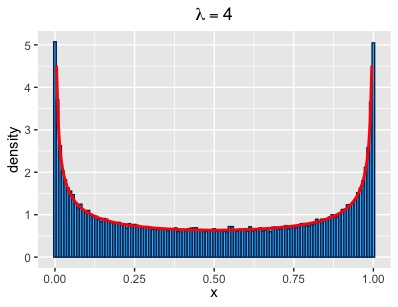}
\caption{Histogram of the beta distribution with $\alpha=\beta=0.5$. The solid line shows the theoretical distribution given in Equation \eqref{eq:betadist}; the bars are a histogram generated by sampling 100,000 initial values uniformly on $(0,1)$ and then iterating 20 times.}
\label{betadist}
\end{center}
\end{figure}

 However, invariant densities for other values of $\lambda$ are not, in general, symmetric.  In \cite{hall95}, Hall and Wolff provided numerical estimates for the invariant distributions for the deterministic logistic map for various values of $\lambda$.  They demonstrated that any continuous invariant density will have poles (singularities) at the boundary of its support,  as well as additional poles at the interior of the support in the case $\lambda < 4$.  They also show that these invariant distributions were not symmetric.  Figure \ref{fig:determinist_density} shows numerical approximations to theses deterministic invariant densities for $\lambda = 3.87$ and $\lambda = 3.95$  These numerical approximations are histograms of 100,000 sampled initial values iterated 20 times.  They indicate the singularities in the theoretical densities described in Remark 5 of \cite{hall95}. These singularities, or poles, correspond to the orbit of the critical point, $f_{\lambda}(1/2)$.  In addition, Hall and Wolff observed that the invariant density will be supported on the interval $(\frac{1}{4}\lambda^2(1-\frac{\lambda}{4}),\frac{1}{4}\lambda)$. This is due to the fact that any point in the interval $(\frac{1}{4}\lambda,1)$ will be mapped into the interval $(0,\frac{1}{4}\lambda^2(1-\frac{\lambda}{4}))$, and any point in $(0,\frac{1}{4}\lambda^2(1-\frac{\lambda}{4})$ will be mapped to a point in $[\frac{1}{4}\lambda^2(1-\frac{\lambda}{4}),\frac{1}{4}\lambda]$.  Therefore, after finitely many iterations, any point $x\in(0,1)$ must be mapped into  $[\frac{1}{4}\lambda^2(1-\frac{\lambda}{4}),\frac{1}{4}\lambda]$, and so the invariant density must be supported on this subinterval.  Note that as $\lambda\rightarrow 4$, $\frac{1}{4}\lambda^2(1-\frac{\lambda}{4})\rightarrow 0$ and $\frac{1}{4}\lambda\rightarrow 1$, so the interval of support widens as $\lambda$ approaches 4. 
 
 \begin{figure}
\begin{center}
 \includegraphics[height =1.8in]{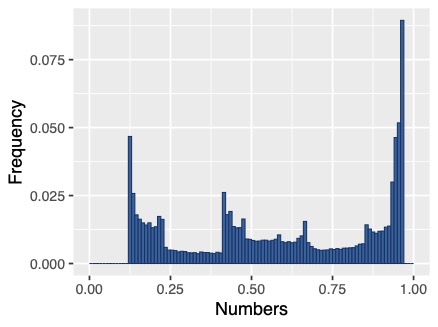}
  \includegraphics[height =1.8in]{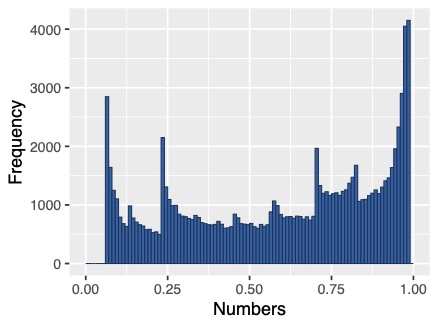}
\caption{Histogram of the simulated invariant distribution for the deterministic logistic map with $\lambda =3.87$ (left) and $\lambda = 3.935$ (right) for 100,000 sampled values, after $n=20$ iterates.}
\label{fig:determinist_density}
\end{center}
\end{figure}

\indent   To numerically explore convergence to the invariant distribution in the stochastic case, we ran several simulations with 100,000 initial values randomly sampled according to a variety of distributions:
    \begin{enumerate}
        \item A discrete distribution supported on the set $\{0.11, 0.33, 0.55, 0.6, 0.78\}$ with probabilities $\{0.196, 0.140, 0.233, 0.322, 0.107\}$ respectively (Figure \ref{fgr:dists}, left).
        \item A discrete distribution supported on $\{0.25,0.5, 0.75\}$  each with probability 1/3.
        \item  A truncated exponential distribution on (0,1) with parameter $r=1.2$ (Figure \ref{fgr:dists}, right).
        \item  A truncated gamma distribution on (0,1) with shape parameter 3 and scale parameter 1.
        \item A truncated normal distribution on (0,1) with mean 0.5 and standard deviation 0.3.
        \item A truncated $t$-distribution on (0,1) with 1 degree of freedom. 
        \item A uniform distribution on (0,1) (Figure \ref{fgr:dists}, center).
    \end{enumerate}
In each case, after several iterations, we obtained a histogram very similar to that presented in Figure \ref{enddist}: a bimodal, asymmetric distribution that appears to be supported throughout the unit interval, without singularities (poles) in the interior.  This is noteworthy, particularly because we simulated using both continuous and discrete initial distributions. We examined distributions that were left-skewed, right-skewed, and symmetric.  Each simulated initial distribution appears to converge to the same distribution, aligning with the global asymptotic stability of the invariant distribution established theoretically earlier in this paper.  We also note that, while the distribution demonstrates peaks toward the sides of the interval, it does not exhibit the poles at the boundary of the interval in the same manner as the invariant distributions for the deterministic maps.  However, this seems reasonable, as the interval of support of the invariant distribution is dependent on $\lambda$, as discussed above, and relatively fewer values of $\lambda$ will map points into ranges closer to 0 and 1, so it makes sense that the distribution will decrease near the boundary of $[0,1]$.  {It is also of note that even as the width of the interval on which $\lambda$ is supported narrows, the invariant distribution maintains this property of decreasing near the boundary of support.  For instance, we ran simulations where $\lambda$ was sampled uniformly on the interval $(3.87,3.9)$ and again on the interval $(3.87,3.935).$  Both invariant distributions exhibited this decrease toward the edge of the boundary of support (which, as expected, is not the entirety of $(0,1)$. Please see Figure} \ref{nar}. \\
\indent As for the speed of convergence, it appears that the convergence happens relatively quickly.  Figure \ref{evol} {shows the evolution of the simulated invariant density after various numbers of iterations, starting from a uniform distribution.} In this set of simulations, the empirical distributions after 5, 20 and 200 iterations are quite close to each other, and the distribution after 20 iterations is almost indistinguishable from the distribution after 200 iterations.  This justifies our use of only 20 iterates in the other simulations presented here.

\begin{figure}[h]
    \begin{center}
    \hbox{
      \includegraphics[height=1.25in]{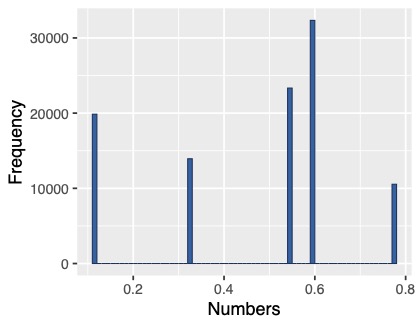} 
      \includegraphics[height = 1.25in]{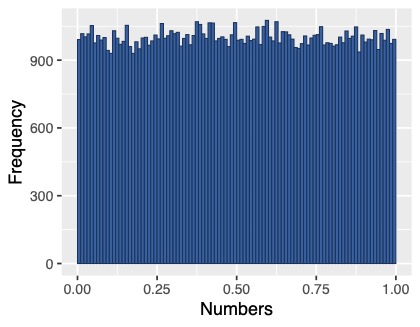} 
      \includegraphics[height = 1.25in]{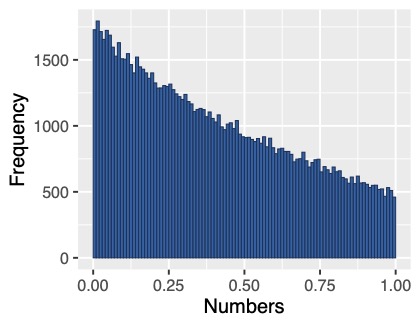}
      }
     \end{center}
    \caption{Starting distributions of 100,000 starting values sampled on the set $\{0.11, 0.33, 0.55, 0.6, 0.78\}$ with probabilities $\{0.196, 0.140, 0.233, 0.322, 0.107\}$ respectively (left), the uniform distribution on (0,1) (middle), and the truncated exponential distribution with parameter $r=1.25$ (right).}
  \label{fgr:dists}
 \end{figure}

 \begin{figure}[ht!]
 \begin{center}
     \includegraphics[width = .6\textwidth]{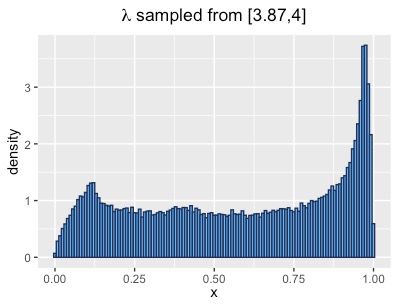}
     \caption{Empirical distribution of $\{X_n\}$ after $n=20$ iterations, where 100,000 initial values were sampled from a truncated exponential distribution with parameter $r=1.25$ as in Figure \ref{fgr:dists}.  Similar results were obtained using other initial distributions (results not shown).}
     \label{enddist}    
 \end{center}
 \end{figure}

 \begin{figure}[ht!]
 \begin{center}

    \includegraphics[height=1.8in]{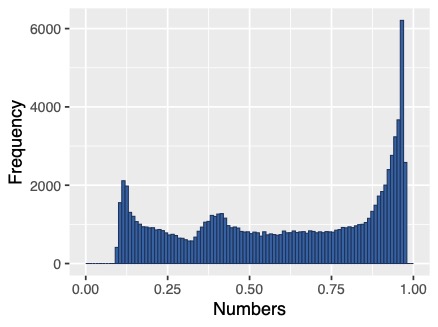}
    \includegraphics[height=1.8in]{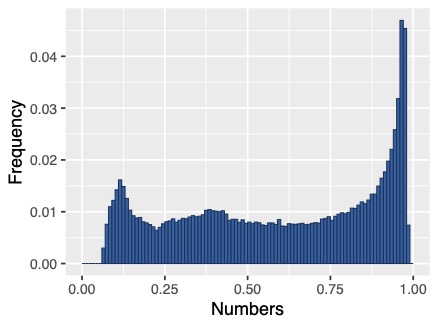}

\caption{The simulated invariant distribution of $\{X_n\}$ for $\lambda\sim U(3.87,3.9)$ (left) and $\lambda\sim U(3.87,3.935)$ (right), shown after 20 iterations.}
    \label{nar}
\end{center}
\end{figure}

\begin{figure}[ht!]
\begin{center}
    \includegraphics[scale=0.8]{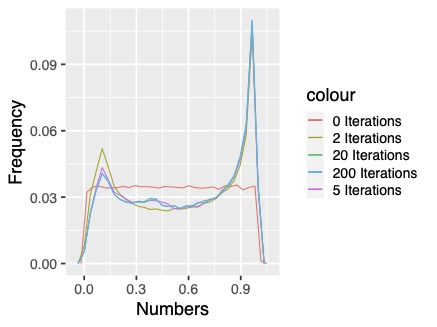}
    \caption{The densities for a simulation of 100,000 values, initially sampled randomly from a uniform distribution on (0,1), after various numbers of iterations where $\lambda$ is uniformly randomly sampled from $(3.87,4)$ each iteration.}
    \label{evol}
\end{center}
\end{figure}
\section{Conclusion}

In conclusion, we have illustrated that the stochastic logistic map, characterized by parameter values uniformly distributed in the chaotic regime, is irreducible. Consequently, the previously established invariant distribution stands as the unique invariant distribution. This result also implies that all initial distributions converge to the invariant distribution in the Ces\`{a}ro sense. Moreover, we have demonstrated the strong aperiodicity of the stochastic logistic map, indicating that initial distributions converge to the invariant distribution with respect to the total variation distance—a more robust result.  We conducted simulations of the stochastic logistic map to visualize the invariant distribution. Remarkably, the empirically derived invariant distribution exhibits qualitative distinctions when compared to the invariant distributions of the deterministic logistic map for values of $\lambda$ in the chaotic regime.  In particular, the simulated distribution does not appear to have a pole at $x=0$.  This contrasts with the work in  \cite{hall95}, which demonstrates   that the deterministic invariant density will have a pole at the boundary of its support.  \\
It should be noted that this work does not seem to depend on the  $\lambda$ values being uniformly distributed on (3.87,4) and indeed, it appears that the results will hold for any absolutely continuous distribution on (3.87,4).

In future work, we plan to provide a theoretical basis for these empirical observations.  In particular, we hope to better describe the asymptotically stable invariant distribution in the case that the $\lambda$-values are uniformly distributed (the case illustrated in Figure \ref{enddist}).  We also plan to provide similar theoretical results when the parameter has other distributions, such as an exponential or discrete distribution.

\bibliography{literature}

\end{document}